\documentclass{article}
\usepackage{amsmath}

\newtheorem{corollary}{Corollary}
\newtheorem{theorem}{Theorem}
\newtheorem{proposition}{Proposition}
\newtheorem{lemma}{Lemma}
\newtheorem{remark}{Remark}
\numberwithin{equation}{section}
\numberwithin{theorem}{section}
\numberwithin{lemma}{section}
\numberwithin{proposition}{section}
\numberwithin{remark}{section}
\newenvironment{proof}[1][Proof]{\noindent\textbf{#1.} }{\ \rule{0.5em}{0.5em}}
\input{tcilatex}
\input{tcilatex}

\begin{document}

\title{Dynamics of ellipses inscribed in quadrilaterals}
\author{Alan Horwitz \\
Professor Emeritus of Mathematics\\
Penn State Brandywine\\
25 Yearsley Mill Rd.\\
Media, PA 19063\\
alh4@psu.edu}
\date{5/5/15}
\maketitle

\begin{abstract}
Let $Q$\ be a convex quadrilateral in the $xy$ plane, let $\limfunc{int}(Q)$
denote the interior of $Q$, and let $\partial (Q)$ denote the boundary of $Q$%
. Let $D_{1}$ and $D_{2}$ denote the diagonals of $Q$\ and let $P$ denote
their point of intersection. Let $P_{0}=(x_{0},y_{0})$ be a point in $\bar{Q}%
=\limfunc{int}(Q)\cup \partial \left( Q\right) $. We prove the following:

(i) If $P_{0}\in \limfunc{int}(Q),P_{0}\notin D_{1}\cup D_{2}$, then there
are exactly two ellipses inscribed in $Q$\ which pass through $P_{0}$.

(ii) If $P_{0}\in \limfunc{int}(Q)$ and $P_{0}\in D_{1}\cup D_{2}$, but $%
P_{0}\neq P$, then there is exactly one ellipse inscribed in $Q$\ which
passes through $P_{0}$.

(iii) There is no ellipse inscribed in $Q$\ which passes through $P$.

(iv) If $P_{0}\in \partial (Q)$, but $P_{0}$ is not one of the vertices of $%
Q $, then there is exactly one ellipse inscribed in $Q$\ which passes
through $P_{0}$(and is thus tangent to $Q$\ at one of its sides).
\end{abstract}

{\LARGE Introduction}

\qquad Suppose that we are given a point, $P_{0}$, in the interior of a
convex quadrilateral, $Q$, in the $xy$ plane. Is there an ellipse, $E$,
inscribed in $Q$ which also passes through $P_{0}$ ? If yes, how many such
ellipses ? By inscribed in $Q$ we mean that $E$ lies in $Q$ and is tangent
to each side of $Q$. Looked at in a dynamic sense: Imagine a particle
constrained to travel along the path of an ellipse inscribed in a convex
quadrilateral, $Q$. Thus the particle bounces off each side of $Q$\ along
its path. Of course there are infinitely many such paths. Can we also
specify a point in $Q$\ that the particle must pass through ? If yes, is
such a path then unique ? We show below(Theorem \ref{T1}) that the path is
unique when $P_{0}$ lies on one of the diagonals of $Q$(but does not equal
their intersection point), while there are two such paths if $P_{0}$ does
not lie on one of the diagonals of $Q$. Finally, if $P_{0}$ equals the
intersection point of the diagonals of $Q$, then no ellipse inscribed in $Q$
passes through $P_{0}$. We also prove that there is a unique ellipse
inscribed in $Q$ which is tangent a given point on the boundary of $Q$,
assuming, of course, that that point is not one of the vertices of $Q$.
Using affine invariance, it suffices to prove Theorem \ref{T1} when $Q$\ is
the convex quadrilateral with vertices $(0,0),(0,1),(1,0)$, and $(s,t)$,
where $s>0,t>0,s+t>1,s\neq 1\neq t$. The proof depends heavily on the
general equation of an ellipse inscribed in $Q$ given in Proposition \ref{P1}
below. It would be interesting to try to extend some of the results of
Theorem \ref{T1} to other families of simple closed convex curves inscribed
in a convex quadrilateral. It seems difficult, though, to come up with the
general equation of such families of curves as we do with ellipses. For a
paper somewhat similar to this one, but involving ellipses inscribed in
triangles, see \cite{H2}.

\section{Main Result}

\begin{theorem}
\label{T1}Let $Q$\ be a convex quadrilateral in the $xy$ plane, let $%
\limfunc{int}(Q)$ denote the interior of $Q$, and let $\partial (Q)$ denote
the boundary of $Q$. Let $D_{1}$ and $D_{2}$ denote the diagonals of $Q$\
and let $P$ denote their point of intersection. Let $P_{0}=(x_{0},y_{0})$ be
a point in $\bar{Q}=\limfunc{int}(Q)\cup \partial \left( Q\right) $.

(i) If $P_{0}\in \limfunc{int}(Q),P_{0}\notin D_{1}\cup D_{2}$, then there
are exactly two ellipses inscribed in $Q$\ which pass through $P_{0}$.

(ii) If $P_{0}\in \limfunc{int}(Q)$ and $P_{0}\in D_{1}\cup D_{2}$, but $%
P_{0}\neq P$, then there is exactly one ellipse inscribed in $Q$\ which
passes through $P_{0}$.

(iii) There is no ellipse inscribed in $Q$\ which passes through $P$.

(iv) If $P_{0}\in \partial (Q)$, but $P_{0}$ is not one of the vertices of $%
Q $, then there is exactly one ellipse inscribed in $Q$\ which passes
through $P_{0}$(and is thus tangent to $Q$\ at one of its sides).
\end{theorem}

\begin{remark}
Instead of just looking at the class of ellipses inscribed in convex
quadrilaterals, $Q$, one might see whether the results of Theorem \ref{T1}
still hold for other families of simple closed convex curves inscribed in $Q$%
. For example, one could start with $x^{n}+y^{n}=1$ and apply all
nonsingular affine transformations to generate such a family. It seems clear
geometrically(we do not have a rigorous proof) that there is no simple
closed convex curve inscribed in $Q$\ which passes through the intersection
point of the diagonals of $Q$. So Theorem \ref{T1}(iii) would still hold for
any family of simple closed convex curves inscribed in $Q$. However, the
other parts of Theorem \ref{T1} would not necessarily hold since they depend
on that particular family of simple closed convex curves.
\end{remark}

By Theorem \ref{T1} we have the following:

\begin{corollary}
If two ellipses inscribed in a convex quadrilateral intersect at a point,
then that point of intersection cannot lie on either diagonal of the
quadrilateral.
\end{corollary}

\section{Preliminary Results}

A problem, often referred to in the literature as Newton's problem, was to
determine the locus of centers of ellipses inscribed in a convex
quadrilateral, $Q$, in the $xy$ plane. Chakerian(\cite{C}) gives a partial
solution of Newton's problem using orthogonal projection, which is the
solution actually given by Newton.

\begin{theorem}
(Newton)\label{Newton}Let $M_{1}$ and $M_{2}$ be the midpoints of the
diagonals of $Q$. If $E$ is an ellipse inscribed in $Q$, then the center of $%
E$ must lie on the open line segment, $Z$, connecting $M_{1}$ and $M_{2}$.
\end{theorem}

In \cite{H1}, \cite{H3}, and \cite{H4} we proved several results about
ellipses inscribed in quadrilaterals. In particular, in \cite{H1} we proved
the following converse of Newton's Theorem.

\begin{theorem}
\label{T0}Let $Q$\ be a convex quadrilateral in the $xy$ plane and let $%
M_{1} $ and $M_{2}$ be the midpoints of the diagonals of $Q$. Let $Z$ be the
open line segment connecting $M_{1}$ and $M_{2}$. If $(h,k)\in Z$, then
there is a unique ellipse with center $(h,k)$ inscribed in $Q$.
\end{theorem}

\begin{remark}
\label{R1}Theorem \ref{T0} gives a one--one map from the interval $I$ to
ellipses inscribed in $Q$.
\end{remark}

We use Theorem \ref{T0} in this paper to derive the general equation of an
ellipse inscribed in $Q$(see Proposition \ref{P1} ).\ We state the following
known result without proof. The first inequality insures that the conic is
an ellipse, while the second insures that the ellipse is non--trivial.

\begin{lemma}
\label{L1}The equation $Ax^{2}+By^{2}+2Cxy+Dx+Ey+F=0$, with $A,B>0$, is the
equation of an ellipse if and only if $AB-C^{2}>0$ and $%
AE^{2}+BD^{2}+4FC^{2}-2CDE-4ABF>0$.
\end{lemma}

We prove Theorem \ref{T1} first when no two sides of $Q$ are parallel and
then when $Q$ is a trapezoid. We do not give all of the details of the proof
when $Q$ is a parallelogram.

\section{No Two Sides Parallel}

Assume first that $Q$ does not have two parallel sides. It suffices, by
affine invariance, to prove Theorem \ref{T1} when $Q$\ is the convex
quadrilateral with vertices $(0,0),(0,1),(1,0)$, and $(s,t)\in G$, where 
\begin{equation*}
G=\left\{ (s,t):s>0,t>0,s+t>1,s\neq 1\neq t\right\} \text{.}
\end{equation*}

Thus we assume throughout the rest of this section that $Q$\ has this form.
It is useful to state some facts about $Q$\ and related notation.

\textbullet\ The sides of $Q$ are given by $S_{1}=\overline{(0,0)\ (1,0)}%
,S_{2}=\overline{(0,0)\ (0,1)},$

$S_{3}=\overline{(s,t)\ (1,0)}$, and $S_{4}=\overline{(0,1)\ (s,t)}$, and
the corresponding lines which make up $\partial {\large (}Q{\large )}$ are
given by $L_{1}$: $y=0,L_{2}$: $x=0,L_{3}$: $y=\dfrac{t}{s-1}(x-1)$, and $%
L_{4}$: $y=1+\dfrac{t-1}{s}x$.

\textbullet\ The diagonals of $Q$\ are $y=\dfrac{t}{s}x$ and $y=1-x$ and
they intersect at $\left( \dfrac{s}{s+t},\dfrac{t}{s+t}\right) $.

\textbullet\ The midpoints of the diagonals of $Q$\ are the points 

$M_{1}=\left( \dfrac{1}{2},\dfrac{1}{2}\right) $ and $M_{2}=\left( \dfrac{1}{%
2}s,\dfrac{1}{2}t\right) $.

\textbullet\ The line through the midpoints of the diagonals of $Q$\ has
equation $y=L(x)$, where 
\begin{equation*}
L(x)=\dfrac{1}{2}\dfrac{s-t+2x(t-1)}{s-1}\text{.}
\end{equation*}

If $s>1$, then

\textbullet\ $\partial \left( Q\right) =\left\{ (x,y):0\leq x\leq
1,y=0\right\} \cup \left\{ (x,y):1\leq x\leq s,y=\dfrac{t}{s-1}(x-1)\right\}
\cup $

$\left\{ (x,y):0\leq x\leq s,y=1+\dfrac{t-1}{s}x\right\} \cup \left\{
(x,y):x=0,0\leq y\leq 1\right\} $.

\textbullet\ $\limfunc{int}(Q)=\left\{ (x,y):0<x\leq 1,0<y<1+\dfrac{t-1}{s}%
x\right\} \cup $

$\left\{ (x,y):1\leq x<s,\dfrac{t}{s-1}(x-1)<y<1+\dfrac{t-1}{s}x\right\} $.

If $s<1$, then

\textbullet\ $\partial \left( Q\right) =\left\{ (x,y):0\leq x\leq
s,y=0\right\} \cup \left\{ (x,y):s\leq x\leq 1,y=\dfrac{t}{s-1}(x-1)\right\}
\cup $

$\left\{ (x,y):0\leq x\leq s,y=1+\dfrac{t-1}{s}x\right\} \cup \left\{
(x,y):x=0,0\leq y\leq 1\right\} $.

\textbullet\ $\limfunc{int}(Q)=\left\{ (x,y):0<x\leq s,0<y<1+\dfrac{t-1}{s}%
x\right\} \cup $

$\left\{ (x,y):s\leq x<1,0<y<\dfrac{t}{s-1}(x-1)\right\} $

\textbullet\ Let $I=\left( \dfrac{1}{2},\dfrac{1}{2}s\right) $ if $s>1$ and $%
I=\left( \dfrac{1}{2}s,\dfrac{1}{2}\right) $ if $s<1$. \ Then any point on
the open line segment connecting $M_{1}$ and $M_{2}$ has the form ${\large (}%
h,L(h){\large )},h\in I$.

\qquad Before giving our first main result, we need the following simple
lemma from \cite{H1}:

\begin{lemma}
\label{L2}If $(s,t)\in G$, then $s+2h(t-1)>0$ for all $h\in I$.
\end{lemma}

Equation (\ref{1}) below was derived using some formulas given in \cite{H1}
for the foci, center, and semi--major and minor axes of an ellipse inscribed
in $Q$. Once the coefficents of the ellipse equation were simplified, we
ended up with a simplified equation, which is given in Proposition \ref{P1}%
(i) below. It is, as expected, much easier to prove that the equation given
below is correct than to give the details of the actual derivation, which we
do not provide here.

\begin{proposition}
\label{P1}(i) $E$ is an ellipse inscribed in $Q$\ if and only if the general
equation of $E$ is given by%
\begin{gather}
4(s-1)^{2}{\large (}L(h){\large )}^{2}\left( x-h\right) ^{2}+4(s-1)^{2}h^{2}%
{\large (}y-L(h){\large )}^{2}  \notag \\
-4(s-1){\large (}2\left( t-1\right) h^{2}+\left( s-t+2\right) h-s{\large )}%
\left( x-h\right) {\large (}y-L(h){\large )}  \label{1} \\
=(2h-1){\large (}2(t-1)h+s{\large )}\left( s-2h\right) ,h\in I\text{.} 
\notag
\end{gather}%
(ii) If $E$ is an ellipse given in (i) for some $h\in I$, then $E$ is
tangent to the four sides of $Q$\ at the points $\zeta _{1}=\left( \dfrac{%
s-2h}{2(t-1)h+s-t},0\right) \in S_{1}$,

$\zeta _{2}=\left( 0,\dfrac{1}{2}\dfrac{s-2h}{\left( s-1\right) h}\right)
\in S_{2}$,

$\zeta _{3}=\left( \dfrac{s+2h(t-1)}{t+s-2h},\dfrac{(2h-1)t^{2}}{\left(
s-1\right) \left( s+t-2h\right) }\right) \allowbreak \in S_{3}$, and

$\zeta _{4}=\left( \dfrac{(2h-1)s^{2}}{s(s+t-2)-2(t-1)h},\dfrac{{\large (}%
s+2h(t-1){\large )}(s-1)}{s(s+t-2)-2(t-1)h}\right) \in S_{4}$.
\end{proposition}

\begin{proof}
First, suppose that $E$ is given by (\ref{1}). Then by simplifying one can
show that $E$ has the form $Ax^{2}+By^{2}+2Cxy+Dx+Ey+F=0$, where $%
A=4(s-1)^{2}{\large (}L(h){\large )}^{2}$, $B=4(s-1)^{2}h^{2}$, $C=-2(s-1)%
{\large (}2\left( t-1\right) h^{2}+\left( s-t+2\right) h-s{\large )}$, $%
D=-2\left( s-2h\right) {\large (}s-t+2h(t-1){\large )}$, $E=\allowbreak
-4h\left( s-1\right) \left( s-2h\right) $, and $F=\left( s-2h\right) ^{2}$.
Then $AB-C^{2}=16(s-1)^{4}h^{2}{\large (}L(h){\large )}^{2}-4(s-1)^{2}%
{\large (}2\left( t-1\right) h^{2}+\left( s-t+2\right) h-s{\large )}%
^{2}=\allowbreak 4\left( 2h-1\right) \left( s-1\right) ^{2}\left(
s-2h\right) {\large (}2(t-1)h+s{\large )}>0$ and $%
AE^{2}+BD^{2}+4FC^{2}-2CDE-4ABF=\allowbreak 16\left( 2h-1\right) ^{2}\left(
s-1\right) ^{2}\left( s-2h\right) ^{2}{\large (}s+2h(t-1){\large )}%
^{2}\allowbreak $ $>0$ by Lemma \ref{L2} and the fact that $h\in I$. Thus,
by Lemma \ref{L1}, (\ref{1}) defines the equation of an ellipse for any $%
h\in I$. Now let $F(x,y)$ equal the left hand side of (\ref{1}). Then $%
F\left( \dfrac{s-2h}{2(t-1)h+s-t},0\right) =\allowbreak F\left( 0,\dfrac{1}{2%
}\dfrac{s-2h}{\left( s-1\right) h}\right) =F\left( \dfrac{s+2h(t-1)}{t+s-2h},%
\dfrac{(2h-1)t^{2}}{\left( s-1\right) \left( s+t-2h\right) }\right)
=\allowbreak $

$F\left( \dfrac{(2h-1)s^{2}}{s(s+t-2)-2(t-1)h},\dfrac{{\large (}s+2h(t-1)%
{\large )}(s-1)}{s(s+t-2)-2(t-1)h}\right) =$

$\allowbreak (2h-1){\large (}2(t-1)h+s{\large )}\left( s-2h\right) $, which
implies that the four points $\zeta _{1}$ thru $\zeta _{4}$ lie on $E$. For
fixed $h$, differentiating both sides of the equation in (\ref{1}) with
respect to $x$ yields: 
\begin{gather*}
8(s-1)^{2}{\large (}L(h){\large )}^{2}\left( x-h\right) +8(s-1)^{2}h^{2}%
{\large (}y-L(h){\large )}\dfrac{dy}{dx} \\
-4(s-1){\large (}2\left( t-1\right) h^{2}+\left( s-t+2\right) h-s{\large )}%
\left( x-h\right) \dfrac{dy}{dx} \\
-4(s-1){\large (}2\left( t-1\right) h^{2}+\left( s-t+2\right) h-s{\large )(}%
y-L(h){\large )}=0\text{,}
\end{gather*}%
so 
\begin{eqnarray*}
\dfrac{dy}{dx} &=&D(x,y)={\large (}4(s-1){\large (}2\left( t-1\right) h^{2}+
\\
&&\left( s-t+2\right) h-s{\large )(}y-L(h){\large )}-8(s-1)^{2}{\large (}L(h)%
{\large )}^{2}\left( x-h\right) {\large )}/ \\
&&{\large (}8(s-1)^{2}h^{2}{\large (}y-L(h){\large )}-4(s-1){\large (}%
2\left( t-1\right) h^{2}+ \\
&&\left( s-t+2\right) h-s{\large )}\left( x-h\right) {\large )}\text{,}
\end{eqnarray*}%
which represents the slope of the ellipse. $D\left( \dfrac{s-2h}{s-2h-t+2ht}%
,0\right) =\allowbreak 0=$ slope of $L_{1}$, $D\left( \dfrac{s-2h+2ht}{t+s-2h%
},\dfrac{2t^{2}h-t^{2}}{\left( s-1\right) \left( t+s-2h\right) }\right)
=\allowbreak \dfrac{t}{s-1}=$ slope of $L_{2}$, and $D\left( \dfrac{%
2s^{2}h-s^{2}}{s^{2}-2s+ts+2h-2ht},\dfrac{\left( s-2h+2ht\right) (s-1)}{%
s^{2}-2s+ts+2h-2ht}\right) =\allowbreak \dfrac{t-1}{s}=$ slope of $L_{3}$.
When $x=0,y=\dfrac{s-2h}{2\left( s-1\right) h}$, the denominator of $D(x,y)$
equals $0$, but the numerator of $D(x,y)$ equals $\allowbreak 2\left(
2h-1\right) \left( s-2h\right) \dfrac{s+2h(t-1)}{h}\neq 0$ by Lemma \ref{L2}
and the fact that $h\in I$. Thus $L_{2}$ is the tangent line at $\zeta _{2}$%
. It follows easily that $\zeta _{1}$ thru $\zeta _{4}$ lie on the line
segments $S_{1}$ thru $S_{4}$, respectively. For any simple closed convex
curve, such as an ellipse, tangent to each side of $Q$ then implies that
that curve lies in $Q$. That proves that $E$ is inscribed in $Q$. Second,
suppose that $E$ is an ellipse inscribed in $Q$. By Theorem \ref{T0},\textbf{%
\ }$E$ has center ${\large (}h_{1},L(h_{1}){\large )}$ for some $h_{1}\in I$%
. We have just shown that (\ref{1}) represents a family of ellipses
inscribed in $Q$\ as $h$ varies over $I$, and each ellipse given by (\ref{1}%
) clearly has center ${\large (}h,L(h){\large )}$ for some $h\in I$. Let $%
\tilde{E}$ be the ellipse given by (\ref{1}) with $h=h_{1}$. Hence $\tilde{E}
$ also has center ${\large (}h_{1},L(h_{1}){\large )}$ and is inscribed in $%
Q $. By Theorem \ref{T0}(see Remark \ref{R1}), $\tilde{E}=E$ and the general
equation of $E$ must be given by (\ref{1}). That proves (i). We have also
just shown that if $E$ is given by (\ref{1}), then $E$ is tangent to the
four sides of $Q$\ at the four points $\zeta _{1}$ thru $\zeta _{4}$, which
proves (ii).
\end{proof}

\begin{lemma}
\label{L3}Let 
\begin{equation}
f(x,y)={\large (}(t-1)x-(s-1)y{\large )}^{2}+2(t-1)x+2(s-1)y+1\text{, }%
(s,t)\in G\text{.}  \label{f}
\end{equation}%
Then $f(x,y)>0$ for any $(x,y)\in \bar{Q}=\limfunc{int}(Q)\cup \partial
\left( Q\right) $.
\end{lemma}

\begin{proof}
While it is obvious that $f(x,y)>0$ in $\bar{Q}$ if both $s$ and $t$ are $>1$%
, we can do the following if one or both of those values is $>1$. First,
expressing $f$ as a quadratic in $s$ we have

$\allowbreak $%
\begin{eqnarray*}
f(x,y) &=&y^{2}s^{2}+\left( -2y^{2}+2y-2xty+2xy\right) s+ \\
&&x^{2}-2x^{2}\allowbreak t+2xty-2xy+x^{2}t^{2}+y^{2}-2x+2xt-2y+1\text{,}
\end{eqnarray*}%
which has discriminant $\left( -2y^{2}+2y-2xty+2xy\right)
^{2}-4y^{2}(x^{2}-2x^{2}\allowbreak
t+2xty-2xy+x^{2}t^{2}+y^{2}-2x+2xt-2y+1)=\allowbreak -16xy^{2}\left(
t-1\right) $. If $t>1$, then the discriminant is negative and thus $f$ has
no roots with $x$ and $y$ real. Since $f\left( \dfrac{s}{s+t},\dfrac{t}{s+t}%
\right) =\allowbreak 4st\dfrac{s+t-1}{\left( s+t\right) ^{2}}>0$, $f(x,y)>0$
in $\bar{Q}$. Similarly, expressing $f$ as a quadratic in $t$ we have 
\begin{eqnarray*}
f(x,y) &=&x^{2}t^{2}+\left( 2xy+2x-2xys-2x^{2}\right) t+ \\
&&x^{2}+1+2\allowbreak xys-2xy-2y^{2}s+y^{2}-2x+y^{2}s^{2}+2ys-2y\text{,}
\end{eqnarray*}

which has discriminant $\left( 2xy+2x-2xys-2x^{2}\right)
^{2}-4x^{2}(x^{2}+1+2\allowbreak
xys-2xy-2y^{2}s+y^{2}-2x+y^{2}s^{2}+2ys-2y)=\allowbreak -16x^{2}y\left(
s-1\right) $. Arguing as above, if $s>1$, then $f(x,y)>0$ in $\bar{Q}$. So
assume now that $s<1$ and $t<1$, and so $I=\left( \dfrac{1}{2}s,\dfrac{1}{2}%
\right) $. $\dfrac{\partial f(x,y)}{\partial x}=0$ and $\dfrac{\partial
f(x,y)}{\partial y}=0$ implies that $(t-1)x+1-(s-1)y=0$ and $%
(t-1)x-1-(s-1)y=0$, which in turn implies that $(t-1)x-(s-1)y=-1$ and $%
(t-1)x-(s-1)y=1$. Thus $f$ has no critical points in $\limfunc{int}(Q)$. To
check $f$ on $\partial \left( Q\right) $: 
\begin{eqnarray*}
f(x,0) &=&\allowbreak {\large (}(t-1)x+1{\large )}^{2},0\leq x\leq s\text{.}
\\
f\left( x,\dfrac{t}{s-1}(x-1)\right) &=&\allowbreak \left( x+t-1\right)
^{2},s\leq x\leq 1\text{.} \\
f\left( x,1+\dfrac{t-1}{s}x\right) &=&\allowbreak \dfrac{{\large (}%
s^{2}+(t-1)x{\large )}^{2}}{s^{2}},0\leq x\leq s\text{.} \\
f(0,y) &=&{\large (}(s-1)y+1{\large )}^{2},0\leq y\leq 1\text{.}
\end{eqnarray*}

$(t-1)x+1=0$ implies that $x=\dfrac{1}{1-t}>1>s$, and thus $(t-1)x+1\neq 0$
if $0\leq x\leq s$, which implies that $f(x,0)>0$.

If $s\leq x\leq 1$, then $x+t-1\geq s+t-1>0$, which implies that $f\left( x,%
\dfrac{t}{s-1}(x-1)\right) >0$.

If $0\leq x\leq s$, then $s^{2}+(t-1)x\geq s^{2}+(t-1)s=s(s+t-1)>0$, which
implies that $f\left( x,1+\dfrac{t-1}{s}x\right) >0$.

Finally, $(s-1)y+1=0$ implies that $y=\dfrac{1}{1-s}>1$, which implies that $%
f(0,y)>0$. Since $f$ has no critical points in $\limfunc{int}(Q)$ and $f$ is
positive on $\partial \left( Q\right) $, $f$ must positive in $\bar{Q}$.
\end{proof}

\begin{proof}
\textbf{(of Theorem \ref{T1} when no two sides are parallel): }For fixed $x$
and $y$, one can rewrite (\ref{1}) in the form $p_{x,y}(h)=0$, where 
\begin{eqnarray*}
p_{x,y}(h) &=&4(s-1)^{2}{\large (}L(h){\large )}^{2}\left( x-h\right)
^{2}+4(s-1)^{2}h^{2}{\large (}y-L(h){\large )}^{2} \\
&&-4(s-1){\large (}2\left( t-1\right) h^{2}+ \\
&&\left( s-t+2\right) h-s{\large )}\left( x-h\right) {\large (}y-L(h){\large %
)} \\
&&-(2h-1){\large (}2(t-1)h+s{\large )}\left( s-2h\right) \text{.}
\end{eqnarray*}%
$p_{x,y}$ depends on $s$ and $t$ as well as $x$ and $y$, but we suppress
that dependence in our notation. Let $f(x,y)$ be given by (\ref{f}). Some
simplification shows that $p_{x,y}$ is really the following quadratic
polynomial in $h$:%
\begin{gather}
p_{x,y}(h)=4f(x,y)h^{2}+  \label{p} \\
4\allowbreak {\large ((}\left( t-1\right) x+1{\large )(}\left( s-t\right) x-s%
{\large )}-\left( s-1\right) y{\large (}s+\left( s-t+2\right) x{\large ))}h+
\notag \\
{\large (}\left( s-t\right) x-s{\large )}^{2}+\allowbreak 4s\left(
s-1\right) xy\text{.}  \notag
\end{gather}
By Proposition \ref{P1}(i), the equation of any ellipse, $E$, inscribed in $%
Q $\ is given by $p_{x,y}(h)=0,h\in I$. Hence, for given $(x_{0},y_{0})$,
the number of times an ellipse $E$ with equation given by (\ref{1}) passes
through $(x_{0},y_{0})$ equals the number of distinct roots of $%
p_{x_{0},y_{0}}(h)=0$ in $I$. Evaluating $p_{x,y}$ at the endpoints of $I$
yields%
\begin{eqnarray}
p_{x,y}\left( \dfrac{1}{2}\right) &=&\left( x+y-1\right) ^{2}\left(
s-1\right) ^{2}\geq 0  \label{endpts} \\
p_{x,y}\left( \dfrac{s}{2}\right) &=&\left( sy-xt\right) ^{2}\allowbreak
\left( s-1\right) ^{2}\geq 0\text{.}  \notag
\end{eqnarray}%
By (\ref{p}), $p_{x,y}^{\prime }(h_{0})=0$, where 
\begin{equation*}
h_{0}=-\dfrac{1}{2}\dfrac{{\large (}\left( t-1\right) x+1{\large )(}\left(
s-t\right) x-s{\large )}-\left( s-1\right) y{\large (}s+\left( s-t+2\right) x%
{\large )}}{{\large (}(t-1)x+1{\large )}^{2}+\left( s-1\right) \allowbreak y%
{\large (}(s-1)y+2-2\left( t-1\right) x{\large )}}\text{.}
\end{equation*}

$h_{0}$ depends on $s,t,x$, and $y$, but we suppress that dependence in our
notation. After some simplification, it follows easily that 
\begin{equation}
p_{x,y}(h_{0})=\dfrac{4\left( s-1\right) ^{2}xy{\large (}\left( x-1\right)
t-(s-1)y{\large )(}\left( t-1\right) x-s(y-1){\large )}}{f(x,y)}\text{.}
\label{ph0}
\end{equation}%
We now assume throughout that $s>1$ and thus $I=\left( \dfrac{1}{2},\dfrac{1%
}{2}s\right) $. The case $s<1$ follows similarly and we omit the details.
Suppose that $(x,y)\in \limfunc{int}(Q)$. We shall prove that $%
p_{x,y}(h_{0})<0$. First, $y<1+\dfrac{t-1}{s}x,0<x<s$ implies that $%
sy-s<\left( t-1\right) x$, so $\left( t-1\right) x-s(y-1)>0$. Second, if $%
1\leq x<s$, then $\dfrac{t}{s-1}(x-1)<y$ since $(x,y)\in \limfunc{int}(Q)$,
while if $0<x<1$, then $\dfrac{t}{s-1}(x-1)<0<y$.

Thus $\left( x-1\right) t-(s-1)y<0,0<x<s$ and it follows that the numerator
of $p_{x,y}(h_{0})$ is negative, while the denominator of $p_{x,y}(h_{0})$
is positive by Lemma \ref{L3}. Summarizing:%
\begin{equation}
p_{x,y}^{\prime }(h_{0})=0\ \text{and }(x,y)\in \limfunc{int}(\text{$Q$}%
)\Rightarrow p_{x,y}(h_{0})<0\text{.}  \label{3}
\end{equation}%
We now prove that $\dfrac{1}{2}<h_{0}<\dfrac{s}{2}$. One could try to prove
this directly, but we found it easier to use the following approach.\ Let $%
S=\left\{ (x,y)\in \limfunc{int}(Q):h_{0}\in I\right\} $. We shall prove
that $S$ is nonempty and both open and closed in $\limfunc{int}(Q)$ in the
Euclidean norm. Since $\limfunc{int}(Q)$ is connected, that will prove that $%
S=\limfunc{int}(Q)$. To prove that $S$ is open, suppose that $(x,y)\in S$,
so that $p_{x,y}^{\prime }(h_{0})=0,\dfrac{1}{2}<h_{0}<\dfrac{s}{2}$.
Suppose also that $\left\{ (x_{k},y_{k})\right\} \in \limfunc{int}(Q$\ $)$
converges to $(x,y)$ in the Euclidean norm, and let $h_{0,k}$ be the root of 
$p_{x_{k},y_{k}}^{\prime }$. Clearly $\dfrac{1}{2}<h_{0,k}<\dfrac{s}{2}$ for
sufficiently large $k$, which implies that $(x_{k},y_{k})$ $\in S$ for
sufficiently large $k$ and thus $S$ is open. To prove that $S$ is closed,
suppose that $(x_{k},y_{k})\in S$ for all $k$ and that $\left\{
(x_{k},y_{k})\right\} $ converges to $(x,y)\in \limfunc{int}(Q)$ in the
Euclidean norm. Again, let $h_{0,k}$ be the root of $p_{x_{k},y_{k}}$.Let $%
\left\{ h_{0,k_{j}}\right\} $ be any convergent subsequence of $\left\{
h_{0,k}\right\} $ with $h_{0,k_{j}}\rightarrow h_{0}$. Since $%
p_{x_{k},y_{k}}^{\prime }(h_{0,k_{j}})=0$ and $\dfrac{1}{2}<h_{0,k_{j}}<%
\dfrac{s}{2}$, $p_{x,y}^{\prime }(h_{0})=0$ and $\dfrac{1}{2}\leq h_{0}\leq 
\dfrac{s}{2}$. If $h_{0}=\dfrac{1}{2}$ or $h_{0}=\dfrac{s}{2}$, then $%
p_{x,y}\left( \dfrac{1}{2}\right) <0$ or $p_{x,y}\left( \dfrac{s}{2}\right)
<0$ by (\ref{3}), which contradicts (\ref{endpts}). Thus $\dfrac{1}{2}<h_{0}<%
\dfrac{s}{2}$ and so $S$ is closed. Finally, to show that $S$ is nonempty,
let $x=\dfrac{s}{s+t}$ and $y=\dfrac{t}{s+t}$. Since the diagonals of $Q$\
intersect at $\left( \dfrac{s}{s+t},\dfrac{t}{s+t}\right) $, $(x,y)\in 
\limfunc{int}(Q)$. The corresponding $p_{x,y}$ is given by $%
p_{x,y}(h)=\allowbreak \dfrac{4st(s+t-1)}{\left( s+t\right) ^{2}}\left(
2h-1\right) \left( 2h-s\right) $, and $p_{x,y}^{\prime }(h)=\dfrac{8st(s+t-1)%
}{\left( s+t\right) ^{2}}(4h-s-1)$, which implies that $h_{0}=\dfrac{1}{4}s+%
\dfrac{1}{4}$. Since $h_{0}-\dfrac{1}{2}=\allowbreak \dfrac{1}{4}(s-1)>0$
and $h_{0}-\dfrac{s}{2}=-\dfrac{1}{4}(s-1)<0,\dfrac{1}{2}<h_{0}<\dfrac{s}{2}$%
. Hence $\left( \dfrac{s}{s+t},\dfrac{t}{s+t}\right) \in S$ and thus $S$ is
nonempty. Hence we have shown that $S=\limfunc{int}(Q)$, which implies that $%
h_{0}\in I$ for any $(x,y)\in \limfunc{int}(Q$$)$. Since $p_{x,y}(h_{0})<0$
and $h_{0}\in I$, by (\ref{endpts}) we have 
\begin{equation}
(x,y)\in \limfunc{int}(\text{$Q$})\Rightarrow p_{x,y}\ \text{has two
distinct roots in }\bar{I}=\left[ \dfrac{1}{2},\dfrac{1}{2}s\right] \text{.}
\label{proots}
\end{equation}%
To prove (i), let $P_{0}\in \limfunc{int}(Q),P_{0}\notin D_{1}\cup D_{2}$.
Then $x_{0}+y_{0}-1\neq 0\neq sy_{0}-tx_{0}$, which implies, by (\ref{endpts}%
), that $p_{x_{0},y_{0}}\left( \dfrac{1}{2}\right) >0$ and $%
p_{x_{0},y_{0}}\left( \dfrac{s}{2}\right) >0$. By (\ref{proots}), $%
p_{x_{0},y_{0}}$ has two distinct roots, $h_{1}$ and $h_{2}$, in $I$. Let $%
E_{j}$ be the ellipse with equation given by (\ref{1}) with $h=h_{j},j=1,2$.
By Proposition \ref{P1}(i), $E_{j}$ is inscribed in $Q$. Since $%
p_{x_{0},y_{0}}(h_{j})=0$, $E_{j}$ passes through $P_{0}$ since (\ref{1})
holds with $x=x_{0},y=y_{0}$, and $h=h_{j},j=1,2$. Note that $E_{1}$ and $%
E_{2}$ are distinct since they have different centers. Now suppose that an
ellipse $\tilde{E}\neq E_{j},j=1,2$ is inscribed in $Q$ and passes through $%
P_{0}$, and suppose that $\tilde{E}$ has center ${\large (}h_{3},L(h_{3})%
{\large )},h_{3}\in I$. By Theorem \ref{T0}(see Remark \ref{R1}), $h_{3}\neq
h_{j}$ for $j=1$ or $2$. By Proposition \ref{P1}(i), the equation of $\tilde{%
E}$ is given by (\ref{1}) with $h=h_{3}$. Since $\tilde{E}$ passes through $%
P_{0}$, $p_{x_{0},y_{0}}(h_{3})=0$, which would give the quadratic $%
p_{x_{0},y_{0}}$ at least three distinct roots. Hence the only ellipses
inscribed in $Q$ which pass through $P_{0}$\ are the ones corresponding to $%
h_{1}$ and $h_{2}$, and so there are precisely two distinct ellipses
inscribed in $Q$\ which pass through $P_{0}$. To prove (ii), let $P_{0}\in
D_{1}\cup D_{2}$, but $P_{0}\neq P=\left( \dfrac{s}{s+t},\dfrac{t}{s+t}%
\right) $. Then either $x_{0}+y_{0}-1=0$ or $sy_{0}-tx_{0}=0$, but not both,
which implies that $p_{x_{0},y_{0}}\left( \dfrac{1}{2}\right) =0$ and $%
p_{x_{0},y_{0}}\left( \dfrac{s}{2}\right) >0$, or $p_{x_{0},y_{0}}\left( 
\dfrac{1}{2}\right) >0$ and $p_{x_{0},y_{0}}\left( \dfrac{s}{2}\right) =0$.
In either case, only one of the roots, say $h_{1}$, of $p_{x_{0},y_{0}}$
lies in $I$, while $h_{2}=\dfrac{1}{2}$ or $\dfrac{s}{2}$, where $h_{2}$ is
the other root of $p_{x_{0},y_{0}}$. Let $E$ be the ellipse with equation
given by (\ref{1}) with $h=h_{1}$. Arguing as above, $E$ is inscribed in $Q$
and passes through $P_{0}$. Now suppose that an ellipse $\tilde{E}\neq E$ is
inscribed in $Q$ and passes through $P_{0}$, and suppose that $\tilde{E}$
has center ${\large (}h_{3},L(h_{3}){\large )},h_{3}\in I$. By Theorem \ref%
{T0}(see Remark \ref{R1}), $h_{3}\neq h_{1}$ and $h_{3}\neq h_{2}$ since $%
h_{2}\notin I$. By Proposition \ref{P1}(i), the equation of $\tilde{E}$ is
given by (\ref{1}) with $h=h_{3}$. Since $\tilde{E}$ passes through $P_{0}$, 
$p_{x_{0},y_{0}}(h_{3})=0$, which again would give the quadratic $%
p_{x_{0},y_{0}}$ at least three distinct roots. Hence the only ellipse
inscribed in $Q$ which passes through $P_{0}$\ is the one corresponding to $%
h_{1}$, and so there is precisely one ellipse inscribed in $Q$\ which passes
through $P_{0}$. To prove (iii), let $P_{0}=P=\left( \dfrac{s}{s+t},\dfrac{t%
}{s+t}\right) $. By (\ref{endpts}), $p_{x_{0},y_{0}}$ vanishes at both
endpoints of $I$, which implies that $p_{x_{0},y_{0}}$ has no roots in $I$.
If there were an ellipse, $E$, inscribed in $Q$ and passing through $P_{0}$,
then the equation of $E$ is given by (\ref{1}) with $h=h_{3}$ for some $%
h_{3}\in I$. But then $p_{x_{0},y_{0}}(h_{3})=0$, which contradicts the fact
that $p_{x_{0},y_{0}}$ has no roots in $I$. Hence there is no ellipse
inscribed in $Q$\ which passes thru $P$. Finally, to prove (iv), suppose
that $P_{0}\in \partial \left( Q\right) $, but $P_{0}$ is not one of the
vertices of $Q$. By taking limits of points $P_{0}\in \limfunc{int}(Q)$, it
follows by (\ref{proots}) that $p_{x_{0},y_{0}}$ has two roots in $\bar{I}$,
which are not necessarily distinct(indeed we shall prove that they are not
distinct). Since $P_{0}\in \partial (Q)$, either $x_{0}=0,y_{0}=0,t\left(
x_{0}-1\right) -(s-1)y_{0}=0$, or $\left( t-1\right) x_{0}-s(y_{0}-1)=0$.
Then by (\ref{ph0}), $p_{x_{0},y_{0}}(h_{0})=0$, which implies that $%
p_{x_{0},y_{0}}$ has a double root at $h_{0}\in \bar{I}$ since $%
p_{x_{0},y_{0}}^{\prime }(h_{0})=0$ by definition. Since $P_{0}$ is not a
vertex of $Q$, $P_{0}$ cannot lie on either diagonal of $Q$. Thus by (\ref%
{endpts}), $p_{x_{0},y_{0}}\left( \dfrac{1}{2}\right) >0$ and $%
p_{x_{0},y_{0}}\left( \dfrac{s}{2}\right) >0$, which implies that $h_{0}\in
I $. Let $E$ be the ellipse with equation given by (\ref{1}) with $h=h_{0}$.
Since $p_{x_{0},y_{0}}(h_{0})=0$, $E$ passes through $P_{0}$. Also, since $%
p_{x_{0},y_{0}}^{\prime }(h_{0})=0$, it is not hard to show that $E$ must be
tangent at $P_{0}$ to a side of $Q$. One could also argue that $E$ must be
inscribed in $Q$ by Proposition \ref{P1}(i). Arguing as above using the fact
that $p_{x_{0},y_{0}}$ cannot have three or more distinct roots, there is no
ellipse $\tilde{E}\neq E$ inscribed in $Q$ and which also passes through $%
P_{0}$, which proves uniqueness.
\end{proof}

\begin{remark}
We could also have used Proposition \ref{P1}(ii) to prove Theorem \ref{T1}%
(iv).
\end{remark}

\textbf{Examples: }(1) $s=\dfrac{1}{2}$, $t=\dfrac{3}{4}$, $x_{0}=\dfrac{1}{3%
}$, $y_{0}=\dfrac{3}{4}$, and $Q$\ is the convex quadrilateral with vertices 
$(0,0),(0,1),(1,0)$, and $\allowbreak \left( \dfrac{1}{2},\dfrac{3}{4}%
\right) $. Then $I=\left( \dfrac{1}{4},\dfrac{1}{2}\right) $ and $%
P_{0}=\left( \dfrac{1}{3},\dfrac{3}{4}\right) \in \limfunc{int}%
(Q),P_{0}\notin D_{1}\cup D_{2}$. By Theorem \ref{T1}(i), there are exactly
two ellipses, $E_{1}$ and $E_{2}$, inscribed in $Q$\ and which pass through $%
P_{0}$. $p_{x_{0},y_{0}}(h)\allowbreak =\dfrac{97}{144}h^{2}-\dfrac{37}{72}h+%
\dfrac{13}{144}$, which has roots $\dfrac{37}{97}\pm \dfrac{6}{97}\sqrt{3}%
\in I$. Letting $h=\dfrac{37}{97}-\dfrac{6}{97}\sqrt{3}$ in (\ref{1}) yields
the equation of $E_{1}$: $\left( \allowbreak 29\,673-4104\sqrt{3}\right)
x^{2}+\left( 23\,632-7104\sqrt{3}\right) y^{2}+\left( 26\,808-25\,488\sqrt{3}%
\right) xy+\left( 18\,864\sqrt{3}-38\,340\right) x+\left( 19\,104\sqrt{3}%
-37\,104\right) y=\allowbreak 9792\sqrt{3}-17\,316$, and letting $h=\dfrac{37%
}{97}+\dfrac{6}{97}\sqrt{3}$ in (\ref{1}) yields the equation of

$E_{2}$: $\left( \allowbreak 29\,673+4104\sqrt{3}\right) x^{2}+\left(
23\,632+7104\sqrt{3}\right) y^{2}+\left( 26\,808+25\,488\sqrt{3}\right)
xy-\left( 18\,864\sqrt{3}+38\,340\right) x$

$-\left( 19\,104\sqrt{3}+37\,104\right) y=-\allowbreak 9792\sqrt{3}-17\,316$.

(2) $s=4$, $t=2$, $x_{0}=\dfrac{1}{2}$, $y_{0}=\dfrac{1}{4}$, and $Q$\ is
the convex quadrilateral with vertices $(0,0),(0,1),(1,0)$, and $(4,2)$.
Then $I=\left( \dfrac{1}{2},2\right) $ and $P_{0}=\left( \dfrac{1}{2},\dfrac{%
1}{4}\right) \in \limfunc{int}(Q)$. Since $sy_{0}-tx_{0}=\allowbreak 0$, $%
P_{0}\in D_{1}\cup D_{2}$, and since $P_{0}\neq P=\allowbreak \left( \dfrac{2%
}{3},\dfrac{1}{3}\right) $, by Theorem \ref{T1}(ii), there is exactly one
ellipse, $E$, inscribed in $Q$\ which passes through $P_{0}$. $%
p_{x_{0},y_{0}}(h)\allowbreak =\dfrac{57}{4}h^{2}\allowbreak -36h+15$, which
has roots $h_{1}=\dfrac{10}{19}\in I$ and $h_{2}=2\in \bar{I}-I$. Letting $h=%
\dfrac{10}{19}$ in (\ref{1}) yields the equation of $E$: $%
63\,916x^{2}+68\,400y^{2}+110\,352xy-123\,424x-\allowbreak
127\,680y+64\,960=5376$.

\section{Trapezoid}

Assume now that $Q$ is a trapezoid. It suffices, by affine invariance, to
prove Theorem \ref{T1} where $Q$\ is the trapezoid\textbf{\ }with vertices $%
(0,0),(1,0),(0,1)$, and $(1,t),0<t\neq 1$. It is useful to state some facts
about $Q$\ and related notation.

\textbullet\ The sides of $Q$ are given by $S_{1}=\overline{(0,0)\ (1,0)}%
,S_{2}=\overline{(0,0)\ (0,1)},S_{3}=\overline{(1,t)\ (1,0)}$, and $S_{4}=%
\overline{(0,1)\ (1,t)}$, and the corresponding lines which make up $%
\partial {\large (}Q{\large )}$ are given by

$L_{1}$: $y=0,L_{2}$: $x=0,L_{3}$: $x=1$, and $L_{4}$: $y=1+(t-1)x$.

\textbullet\ The diagonals of $Q$\ are $y=tx$ and $y=1-x$, and they
intersect at $\left( \dfrac{1}{1+t},\dfrac{t}{1+t}\right) $.

\textbullet\ The midpoints of the diagonals of $Q$\ are the points $%
M_{1}=\left( \dfrac{1}{2},\dfrac{1}{2}\right) $ and $M_{2}=\left( \dfrac{1}{2%
},\dfrac{1}{2}t\right) $.

\textbullet\ The open line segment joining $M_{1}$ and $M_{2}$ is $\left\{ h=%
\dfrac{1}{2},k\in I\right\} $, where $I=\left( \dfrac{1}{2},\dfrac{1}{2}%
t\right) $ if $t>1$, or $I=\left( \dfrac{1}{2}t,\dfrac{1}{2}\right) $ if $%
t<1 $.

\textbullet\ $\partial \left( Q\right) =\left\{ (x,y):0\leq x\leq
1,y=0\right\} \cup \left\{ (x,y):x=1,0\leq y\leq t\right\} $

$\cup \left\{ (x,y):0\leq x\leq 1,y=1+(t-1)x\right\} \cup \left\{
(x,y):x=0,0\leq y\leq 1\right\} $ and

$\limfunc{int}\left( Q\right) =\left\{ (x,y):0<x<1,0<y<1+(t-1)x\right\} $.

\textbullet\ Note that 
\begin{equation}
(t-1)x+1>0,0\leq x\leq 1\text{.}  \label{5}
\end{equation}

First we need the following proposition, which is the analogy of Proposition %
\ref{P1} for trapezoids.

\begin{proposition}
\label{P2}Let $Q$\ be the trapezoid\textbf{\ }with vertices $%
(0,0),(1,0),(0,1)$, and $(1,t),0<t\neq 1$.

(i) $E$ is an ellipse inscribed in $Q$\ if and only if the general equation
of $E$ is given by 
\begin{gather}
4k^{2}\left( t-1\right) ^{2}x^{2}+(t-1)^{2}y^{2}+\allowbreak 4\left(
1-t\right) \left( tk-t+k\right) xy  \notag \\
+4k\left( t-1\right) \left( 2k-t\right) x+\allowbreak 2\left( t-1\right)
\left( 2k-t\right) y  \label{4} \\
=-\left( 2k-t\right) ^{2},k\in I\text{.}  \notag
\end{gather}%
(ii) If $E$ is the ellipse given in (i), then $E$ is tangent to the four
sides of $Q$\ at the points $\zeta _{1}=\left( \dfrac{2k-t}{2k\left(
1-t\right) },0\right) \in S_{1}$, $\zeta _{2}=\left( 0,\dfrac{2k-t}{1-t}%
\right) \in S_{2}$, $\zeta _{3}=\left( 1,t\dfrac{1-2k}{1-t}\right)
\allowbreak \in S_{3}$, and $\zeta _{4}=\left( \dfrac{1-2k}{(1-t)(t+1-2k)},%
\dfrac{t}{t+1-2k}\right) \in S_{4}$.
\end{proposition}

\begin{proof}
First, suppose that $E$ is given by (\ref{4}). Then $E$ has the form $%
Ax^{2}+By^{2}+2Cxy+Dx+Ey+F=0$, where $A=4k^{2}\left( t-1\right) ^{2}$, $%
B=(t-1)^{2}$, $C=2\left( 1-t\right) \left( tk-t+k\right) $, $D=4k\left(
t-1\right) \left( 2k-t\right) $, $E=\allowbreak 2\left( t-1\right) \left(
2k-t\right) $, and $F=\left( 2k-t\right) ^{2}$. Then $AB-C^{2}=\allowbreak
4t\left( 2k-1\right) \left( t-2k\right) \left( t-1\right) ^{2}>0$ since $%
k\in I$ implies that $t<2k<1$ if $t<1$, and $1<2k<t$ if $t>1$. Also, $%
AE^{2}+BD^{2}+4FC^{2}-2CDE-4ABF=\allowbreak 16t^{2}\left( t-1\right)
^{2}\left( 2k-1\right) ^{2}\left( 2k-t\right) ^{2}>0$ since $t<2k<1$ if $t<1$%
, and $1<2k<t$ if $t>1$. By Lemma \ref{L1}, (\ref{4}) defines the equation
of an ellipse for any $k\in I$. Now let $F(x,y)=4k^{2}\left( t-1\right)
^{2}x^{2}+(t-1)^{2}y^{2}+\allowbreak 4\left( 1-t\right) \left( tk-t+k\right)
xy+4k\left( t-1\right) \left( 2k-t\right) x+\allowbreak 2\left( t-1\right)
\left( 2k-t\right) y+\allowbreak \left( 2k-t\right) ^{2}$, the left hand
side of (\ref{4}). Then $F\left( \dfrac{2k-t}{2k\left( 1-t\right) },0\right)
=\allowbreak F\left( 0,\dfrac{2k-t}{1-t}\right) =F\left( 1,t\dfrac{1-2k}{1-t}%
\right) =F\left( \dfrac{1-2k}{(1-t)(t+1-2k)},\dfrac{t}{t+1-2k}\right)
=-\left( 2k-t\right) ^{2}$, which implies that the four points $\zeta _{1}$
thru $\zeta _{4}$ lie on $E$. Differentiating both sides of the equation in (%
\ref{4}) with respect to $x$ yields $8k^{2}\left( t-1\right)
^{2}x+2(t-1)^{2}y\dfrac{dy}{dx}+\allowbreak 4\left( 1-t\right) \left(
tk-t+k\right) y+\allowbreak 4\left( 1-t\right) \left( tk-t+k\right) x\dfrac{%
dy}{dx}+4k\left( t-1\right) \left( 2k-t\right) +\allowbreak 2\left(
t-1\right) \left( 2k-t\right) \dfrac{dy}{dx}=0$, which implies that $\dfrac{%
dy}{dx}=D(x,y)=\dfrac{4k^{2}\left( t-1\right) x-\allowbreak 2\left(
tk-t+k\right) y+2k\left( 2k-t\right) }{2\left( tk-t+k\right)
x-(t-1)y-\allowbreak \left( 2k-t\right) }$. Now $D\left( \dfrac{2k-t}{%
2k\left( 1-t\right) },0\right) =\allowbreak 0=$ slope of $L_{1}$ and $%
D\left( \dfrac{1-2k}{(1-t)(t+1-2k)},\dfrac{t}{t+1-2k}\right) =\allowbreak
t-1=$ slope of $L_{4}$.

When $x=0$ and $y=\dfrac{2k-t}{1-t}$, the denominator of $D(x,y)$ equals $0$%
, but the numerator of $D(x,y)$ equals $\dfrac{2\allowbreak t\left(
2k-t\right) (2k-1)}{t-1}\neq 0$. Thus $L_{2}$ is the tangent line at $\zeta
_{2}$. When $x=1$ and $y=t\dfrac{1-2k}{1-t}$, again the denominator of $%
D(x,y)$ equals $0$, but the numerator of $D(x,y)$ equals$\dfrac{\allowbreak
2t\left( 1-2k\right) \left( 2k-t\right) }{t-1}\neq 0$.$\allowbreak $ Thus $%
L_{3}$ is the tangent line at $\zeta _{3}$. It follows easily that $\zeta
_{1}$ thru $\zeta _{4}$ lie on the line segments $S_{1}$ thru $S_{4}$,
respectively. For any simple closed convex curve, such as an ellipse,
tangent to each side of $Q$ then implies that that curve lies in $Q$. That
proves that $E$ is inscribed in $Q$. Second, suppose that $E$ is an ellipse
inscribed in $Q$. By Theorem \ref{T0},\textbf{\ }$E$ has center $\left( 
\dfrac{1}{2},k\right) $ for some $k\in I$. We have just shown that (\ref{4})
represents a family of ellipses inscribed in $Q$\ as $k$ varies over $I$,
and it is not hard to show that each ellipse given by (\ref{4}) has center $%
\left( \dfrac{1}{2},k\right) $ for some $k\in I$. Arguing exactly as in the
case above when $Q$ does not have two parallel sides, it follows that the
general equation of $E$ must be given by (\ref{4}). That proves (i). We have
also just shown that if $E$ is given by (\ref{4}), then $E$ is tangent to
the four sides of $Q$\ at the four points $\zeta _{1}$ thru $\zeta _{4}$,
which proves (ii).\qquad
\end{proof}

\begin{lemma}
\textbf{\label{L4}}Let $g(x,y)=(x-1)^{2}+(y-tx)(x-1)+txy$ and $%
h(x,y)=t(t-1)x^{2}-(t+1)xy+tx+y$. Then $g(x,y)>0$ and $h(x,y)>0$ if $%
(x,y)\in \limfunc{int}(Q)$.
\end{lemma}

\begin{proof}
Suppose that $(x,y)\in \limfunc{int}(Q)$. Then $0<y<1+(t-1)x$, which implies
that $g(x,y)>(x-1)^{2}+{\large (}1+(t-1)x-tx{\large )}(x-1)+tx=\allowbreak
tx>0$.

Now express $h$ as a linear function of $y,L(y)$, for fixed $x$ and $t$: $%
L(y)={\large (}1-(t+1)x{\large )}y+t(t-1)x^{2}+tx$. Then by (\ref{5}), $%
L(0)=t(t-1)x^{2}+tx=tx{\large (}(t-1)x+1{\large )}>0$ and $%
L(1+(t-1)x)=\allowbreak \left( 1-x\right) {\large (}(t-1)x+1{\large )}>0$,
which implies that $L(y)>0$ for all $0<y<1+(t-1)x$.

\textbf{(of Theorem \ref{T1} for trapezoids)}For fixed $x$ and $y$, one can
write (\ref{4}) in the form $p_{x,y}(k)=0$, where $p_{x,y}$ is the
polynomial in $k$ given by

\begin{gather}
p_{x,y}(k)=4\left( xt-x+1\right) ^{2}k^{2}  \notag \\
+4(\allowbreak tx-yxt^{2}-xt^{2}-y+yt-t+yx)k  \label{pk} \\
+(t+y)^{2}+ty\left( yt-2t-2y\right) +4txy\left( t-1\right) \text{.}  \notag
\end{gather}

$p_{x,y}$ depends on $t$ as well, but we suppress that dependence in our
notation. By Proposition \ref{P2}(i), the equation of any ellipse, $E$,
inscribed in $Q$\ is given by $p_{x,y}(k)=0,k\in I$. Hence, for given $(x,y)$%
, the number of times an ellipse $E$ with equation given by (\ref{4}) passes
through $(x,y)$ equals the number of distinct roots of $p_{x,y}(k)=0$ in $I$%
. Evaluating $p_{x,y}$ at the endpoints of $I$ yields 
\begin{eqnarray}
p_{x,y}\left( \dfrac{1}{2}\right) &=&\allowbreak \left( y+x-1\right)
^{2}\left( t-1\right) ^{2},  \label{6} \\
p_{x,y}\left( \dfrac{t}{2}\right) &=&\allowbreak \left( t-1\right)
^{2}\left( tx-y\right) ^{2}\text{.}  \notag
\end{eqnarray}%
Now $p_{x,y}^{\prime }(k_{0})=\allowbreak 0$, where 
\begin{equation}
k_{0}=\dfrac{(t^{2}-1)xy+t(t-1)x-(t-1)y+t}{2{\large (}(t-1)x+1{\large )}^{2}}%
\text{.}  \label{k0}
\end{equation}%
We now assume throughout that $t>1$ and thus $I=\left( \dfrac{1}{2},\dfrac{1%
}{2}t\right) $. The case $t<1$ follows similarly and we omit the details. To
prove (i), let $P_{0}\in \limfunc{int}(Q),P_{0}\notin D_{1}\cup D_{2}$. A
simple computation yields \ 
\begin{equation}
p_{x_{0},y_{0}}(k_{0})=\dfrac{4t\left( t-1\right) ^{2}x_{0}\left(
x_{0}-1\right) y_{0}{\large (}(t-1)x_{0}-y_{0}+1{\large )}}{{\large (}%
(t-1)x_{0}+1{\large )}^{2}}\text{.}  \label{pk0}
\end{equation}%
Note that the denominator in (\ref{pk0}) is nonzero by (\ref{5}). Now $%
(x_{0},y_{0})\in \limfunc{int}(Q)$ implies that $(t-1)x_{0}-y_{0}+1>0$ and $%
x_{0}-1<0$. Hence $p_{x_{0},y_{0}}(k_{0})<0$. Unlike the case above when $Q$
does not have two parallel sides, here we shall prove directly that $\dfrac{1%
}{2}<k_{0}<\dfrac{t}{2}$. First, $k_{0}-\dfrac{1}{2}=\dfrac{%
(t^{2}-1)x_{0}y_{0}+t(t-1)x_{0}-(t-1)y_{0}+t}{2{\large (}(t-1)x_{0}+1{\large %
)}^{2}}-\dfrac{1}{2}=\allowbreak \allowbreak $

$\left( t-1\right) \dfrac{(x_{0}-1)^{2}+(x_{0}-1)y_{0}-\allowbreak
tx_{0}\left( x_{0}-y_{0}-1\right) }{2{\large (}(t-1)x_{0}+1{\large )}^{2}}%
=\left( t-1\right) g(x_{0},y_{0})>0$ by Lemma \ref{L4}. Second, $\dfrac{t}{2}%
-k_{0}=\dfrac{t}{2}-\dfrac{(t^{2}-1)x_{0}y_{0}+t(t-1)x_{0}-(t-1)y_{0}+t}{2%
{\large (}(t-1)x_{0}+1{\large )}^{2}}=$

$\left( t-1\right) \dfrac{t(t-1)x_{0}^{2}-(t+1)x_{0}y_{0}+tx_{0}+y_{0}}{2%
{\large (}(t-1)x_{0}+1{\large )}^{2}}=\left( t-1\right) h(x_{0},y_{0})>0$,
again by Lemma \ref{L4}. That proves that $\dfrac{1}{2}<k_{0}<\dfrac{t}{2}$
when $P_{0}\in \limfunc{int}(Q)$. Also, $P_{0}\notin D_{1}\cup D_{2}$
implies that $y_{0}+x_{0}-1\neq 0\neq tx_{0}-y_{0}$, which implies, by (\ref%
{6}), that $p_{x_{0},y_{0}}\left( \dfrac{1}{2}\right) >0$ and $%
p_{x_{0},y_{0}}\left( \dfrac{t}{2}\right) >0$. Thus $p_{x_{0},y_{0}}$ must
have two distinct roots in $I$. Arguing as we did for the case when no two
sides of $Q$ are parallel, it follows that precisely two distinct ellipses
inscribed in $Q$\ pass through $P_{0}$. To prove (ii), let $P_{0}\in
D_{1}\cup D_{2}$, but $P_{0}\neq P=\left( \dfrac{1}{1+t},\dfrac{t}{1+t}%
\right) $. Then either $x_{0}+y_{0}-1=0$ or $tx_{0}-y_{0}=0$, but not both,
which implies, by (\ref{6}), that $p_{x_{0},y_{0}}\left( \dfrac{1}{2}\right)
=0$ and $p_{x_{0},y_{0}}\left( \dfrac{t}{2}\right) >0$, or $%
p_{x_{0},y_{0}}\left( \dfrac{1}{2}\right) >0$ and $p_{x_{0},y_{0}}\left( 
\dfrac{t}{2}\right) =0$. In either case $p_{x_{0},y_{0}}$ has exactly one
real root in $I$ since $p_{x_{0},y_{0}}(k_{0})<0$. Again, arguing as we did
for the case when no two sides of $Q$ are parallel, it follows that
precisely one distinct ellipse inscribed in $Q$\ passes through $P_{0}$. To
prove (iii), let $P_{0}=P=\left( \dfrac{1}{1+t},\dfrac{t}{1+t}\right) $.
Then $x_{0}+y_{0}-1=tx_{0}-y_{0}=0$, which implies, by (\ref{6}), that $%
p_{x_{0},y_{0}}\left( \dfrac{1}{2}\right) =p_{x_{0},y_{0}}\left( \dfrac{t}{2}%
\right) =0$ and hence $p_{x_{0},y_{0}}$ has no roots in $I$. As above, there
is no ellipse inscribed in $Q$\ which passes thru $P_{0}$. Finally we prove
(iv). Suppose that $P_{0}\in \partial \left( Q\right) $, but $P_{0}$ is not
one of the vertices of $Q$. Since $P_{0}\in \partial (Q)$, either $%
x_{0}=0,x_{0}=1,y_{0}=0$, or $(t-1)x_{0}-y_{0}+1=0$. Then by (\ref{pk0}), $%
p_{x_{0},y_{0}}(k_{0})=0$, which implies that $p_{x_{0},y_{0}}$ has a double
root at $k_{0}\in \bar{I}$ since $p_{x_{0},y_{0}}^{\prime }(k_{0})=0$ by
definition. Since $P_{0}$ is not a vertex of $Q$, $P_{0}$ cannot lie on
either diagonal of $Q$. Thus by (\ref{6}), $p_{x_{0},y_{0}}\left( \dfrac{1}{2%
}\right) >0$ and $p_{x_{0},y_{0}}\left( \dfrac{t}{2}\right) >0$, which
implies that $k_{0}\in I$. Let $E$ be the ellipse with equation given by (%
\ref{4}) with $k=k_{0}$. Since $p_{x_{0},y_{0}}(k_{0})=0$, $E$ passes
through $P_{0}$. Also, since $p_{x_{0},y_{0}}^{\prime }(k_{0})=0$, it is not
hard to show that $E$ must be tangent at $P_{0}$ to a side of $Q$. One could
also argue that $E$ must be inscribed in $Q$ by Proposition \ref{P2}(i).
Using the fact that $p_{x_{0},y_{0}}$ cannot have three or more distinct
roots, it follows easily that there is no ellipse $\tilde{E}\neq E$
inscribed in $Q$ and which also passes through $P_{0}$, which proves
uniqueness. That completes the proof of Theorem \ref{T1} when $Q$ is a
trapezoid.
\end{proof}

\textbf{Example: }Consider the trapezoid, $\tilde{Q}$, with vertices $%
(-1,2),(3,4),(3,-1)$, and $(9,2)$. We want to find an an ellipse inscribed
in $\tilde{Q}$ which passes through the point $(4,2)\limfunc{int}(\tilde{Q})$%
. First define the affine transformation

$A(x,y)=\left( \dfrac{1}{10}x-\dfrac{1}{5}y+\dfrac{1}{2},\dfrac{3}{20}x+%
\dfrac{1}{5}y-\dfrac{1}{4}\right) $. Then $A(-1,2)=\allowbreak \left(
0,0\right) $, $A(3,4)=\allowbreak \left( 0,1\right) $, $A(3,-1)=\allowbreak
\left( 1,0\right) $, and $A(9,2)=\allowbreak \left( 1,\dfrac{3}{2}\right) $,
so that $A$ maps $\tilde{Q}$ onto the trapezoid, $Q$, with vertices $%
(0,0),(1,0),(0,1)$, and $\left( 1,\dfrac{3}{2}\right) $. Also, $%
A(4,2)=\allowbreak \left( \dfrac{1}{2},\dfrac{3}{4}\right) =P_{0}\in 
\limfunc{int}(Q)$. Now $y_{0}=tx_{0}$, where $x_{0}=\dfrac{1}{2}$, $y_{0}=%
\dfrac{3}{4}$, and $t=\dfrac{3}{2}$, which implies that $P_{0}$ lies on a
diagonal of $Q$, but $P_{0}\neq P=\left( \dfrac{2}{5},\dfrac{3}{5}\right) $.
By Theorem \ref{T1}(ii), there is exactly one ellipse, $E$, inscribed in $Q$%
\ which passes through $P_{0}$. The corresponding polynomial in $%
k,p_{x_{0},y_{0}}$(see (\ref{pk}) above), is given by $p_{x_{0},y_{0}}(k)=%
\dfrac{1}{64}\left( 100k-51\right) \left( 4k-3\right) $, which has one root
in $I$, $k_{1}=\dfrac{51}{100}$. Letting $k=\dfrac{51}{100}$ in (\ref{4})
above yields the equation of $E$: $2601x^{2}+2500y^{2}+4500xy-4896x-4800%
\allowbreak y=-2304$. Finally, substituting $\dfrac{1}{10}x-\dfrac{1}{5}y+%
\dfrac{1}{2}$ for $x$ and $\dfrac{3}{20}x+\dfrac{1}{5}y-\dfrac{1}{4}$for $y$
in the equation of $E$, and simplifying yields the following equation of the
ellipse inscribed in $\tilde{Q}$ which passes through the point $(4,2)$: $%
\allowbreak 3744x^{2}+601y^{2}+24xy-22\,800x-1900y=-\allowbreak 32\,500$.

\section{Parallelogram}

Finally, assume that $Q$ is a parallelogram. Since any parallelogram is
affine equivalent to the unit square, it suffices to assume that $Q$ is the
unit square. In [\cite{H4}] it was shown that if $Z$\ is the rectangle with
vertices $(0,0),(l,0),(0,k)$, and $(l,k)$, where $l,k>0$, then the general
equation of an ellipse inscribed in $Z$\ is given by $%
k^{2}x^{2}+l^{2}y^{2}-2l\left( k-2v\right)
xy-2lkvx-2l^{2}vy+l^{2}v^{2}=0,0<v<k$. It then follows immediately(or one
can easily prove this directly) that the general equation of an ellipse
inscribed in the unit square is given by $x^{2}+y^{2}+2\left( 2v-1\right)
xy-2vx-2vy+v^{2}=0,0<v<1$. The rest of the proof of Theorem \ref{T1} is
similar to the proof for the trapezoid case and we omit the details.

\section{Algorithms}

Throughout, $A$ denotes a non--singular affine map of the $xy$ plane to
itself.

\textbullet\ Given a point, $\tilde{P}_{0}$, in the interior of a
quadrilateral, $\tilde{Q}$, find all ellipses inscribed in $\tilde{Q}$ which
pass through $\tilde{P}_{0}$.

\textbf{Case 1:} $\tilde{Q}$ does not have two parallel sides.

\textbf{Step 1:} Find an $A$ which maps $\tilde{Q}$ to the convex
quadrilateral, $Q$, with vertices $(0,0),(0,1),(1,0)$, and $(s,t)$, where $%
s>0,t>0,s+t>1,s\neq 1\neq t$. Let $I=\left( \dfrac{1}{2},\dfrac{1}{2}%
s\right) $ if $s>1$, $I=\left( \dfrac{1}{2}s,\dfrac{1}{2}\right) $ if $s<1$. 
$A$ sends $\tilde{P}_{0}$ to the point $P_{0}=(x_{0},y_{0})$ lying in the
interior of $Q$.

\textbf{Step 2: }Find the roots of the quadratic polynomial in $h$ given in (%
\ref{p}), with $x=x_{0}$ and $y=y_{0}$. If $P_{0}$ does not lie on either
diagonal of $Q$, then $p_{x_{0},y_{0}}$ has two roots in $I$, $h_{1}$ and $%
h_{2}$. There are then two distinct ellipses, $E_{1}$ and $E_{2}$, inscribed
in $Q$ which pass through $P_{0}$. The equations of $E_{1}$ and $E_{2}$ are
given by (\ref{1}) with $h=h_{j},j=1,2$.

If $P_{0}$ lies on one of the diagonals of $Q$, but $P_{0}$ does not equal
the intersection point of the diagonals, then $p_{x_{0},y_{0}}$ has one root
in $I$, $h_{1}$. There is then one ellipse, $E$, inscribed in $Q$ which
passes through $P_{0}$. The equation of $E$ is given by (\ref{1}) with $%
h=h_{1}$.

In each case above, use the map $A$ to obtain the corresponding equation of
each ellipse inscribed in $\tilde{Q}$ which passes through $\tilde{P}_{0}$.

Finally, if $\tilde{P}_{0}$ equals the intersection point of the diagonals
of $\tilde{Q}$, then there is no ellipse inscribed in $\tilde{Q}$ which
passes through $\tilde{P}_{0}$.

\textbf{Case 2:} $\tilde{Q}$ is a trapezoid.

\textbf{Step 1:} Find an $A$ which maps $\tilde{Q}$ to the trapezoid, $Q$,%
\textbf{\ }with vertices $(0,0),(1,0),(0,1)$, and $(1,t),0<t\neq 1$. Let $%
I=\left( \dfrac{1}{2},\dfrac{1}{2}t\right) $ if $t>1$, or $I=\left( \dfrac{1%
}{2}t,\dfrac{1}{2}\right) $ if $t<1$. $A$ sends $\tilde{P}_{0}$ to the point 
$P_{0}=(x_{0},y_{0})$ lying in the interior of $Q$.

\textbf{Step 2: }Find the roots of the quadratic polynomial in $k$ given in (%
\ref{pk0}), with $x=x_{0}$ and $y=y_{0}$.

As with case 1, if $P_{0}$ does not lie on either diagonal of $Q$, then $%
p_{x_{0},y_{0}}$ has two roots in $I$, $k_{1}$ and $k_{2}$. There are then
two distinct ellipses, $E_{1}$ and $E_{2}$, inscribed in $Q$ which pass
through $P_{0}$. The equations of $E_{1}$ and $E_{2}$ are given by (\ref{4})
with $k=k_{j},j=1,2$.

If $P_{0}$ lies on one of the diagonals of $Q$, but $P_{0}$ does not equal
the intersection point of the diagonals, then $p_{x_{0},y_{0}}$ has one root
in $I$, $k_{1}$. There is then one ellipse, $E$, inscribed in $Q$ which
passes through $P_{0}$. The equation of $E$ is given by (\ref{4}) with $%
k=k_{1}$.

In each case above, use the map $A$ to obtain the corresponding equation of
each ellipse inscribed in $\tilde{Q}$ which passes through $\tilde{P}_{0}$.

Finally, if $\tilde{P}_{0}$ equals the intersection point of the diagonals
of $\tilde{Q}$, then there is no ellipse inscribed in $\tilde{Q}$ which
passes through $\tilde{P}_{0}$.

\textbf{Case 3:} $\tilde{Q}$ is a parallelogram.

The details here are similar to the two cases above. One can first use an
affine map to send $\tilde{Q}$ to the unit square.

\end{document}